\documentclass[11pt]{article}
\usepackage[T1]{fontenc}
\usepackage[utf8]{inputenc}
\usepackage[margin=1.08in]{geometry}
\usepackage{amsmath,amssymb,amsthm,mathtools}
\usepackage{enumitem}
\usepackage{xcolor}
\usepackage[colorlinks=true,linkcolor=blue,citecolor=blue,urlcolor=blue]{hyperref}
\usepackage[nameinlink,noabbrev,capitalize]{cleveref}
\usepackage[symbol]{footmisc}
\numberwithin{equation}{section}
\counterwithout{equation}{section}
\newtheorem{theorem}{Theorem}[section]
\newtheorem{problem}[theorem]{Problem}
\newtheorem{corollary}[theorem]{Corollary}
\newtheorem{lemma}[theorem]{Lemma}
\newcommand{\Z}{\mathbb Z}
\newcommand{\cT}{\mathcal T}
\newcommand{\doi}[1]{\href{https://doi.org/#1}{doi:#1}}

\title{A note on zero-sum Ramsey numbers of complete graphs modulo 3}
\author{Cheng Chi$^\ast$ 
\and Jialin He$^\dagger$
\and Fuhong Ma$^\ddagger$}
\date{}

\begin{document}
\maketitle
\footnotetext[1]{School of Mathematical Sciences, Shanghai Jiao Tong University, 800 Dongchuan Road, Shanghai 200240, China.
    Email: chengchi@sjtu.edu.cn. Supported by National Key R\&D Program of China under grant No. 2022YFA1006400 and National Natural Science Foundation of China No. 12571376.}

\footnotetext[2]{School of Mathematical Sciences, Key Laboratory of MEA (Ministry of Education) and Shanghai Key Laboratory of PMMP, East China Normal University, Shanghai 200241, China.
    Email: jlhe@math.ecnu.edu.cn. Supported in part by Science and Technology Commission of Shanghai Municipality No. 22DZ2229014.}
    
\footnotetext[3]{Corresponding author.School of Mathematics and Statistics, Shandong University of Technology, Zibo 255000, China. Email: mafuhongsdnu@163.com. Supported by NNSF of China (No.12401458) and NSF of Shandong Province (No.ZR2024QA239). 
}

\begin{abstract}
For a graph $H$ with $3\mid e(H)$, the zero-sum Ramsey number $R(H,\Z_3)$ is the least integer $N$ such that every labeling of the edges of $K_N$ by elements of $\Z_3$ contains a copy of $H$ whose edge labels sum to zero.
We determine the last previously unresolved infinite family in the complete-graph case modulo $3$. More precisely, 
we prove that 
\(R(K_n,\Z_3)=n+3\) 
for every $n\ge 10$ satisfying $n\equiv 1\pmod 3$.  
Consequently, for $k\ge 1$,
\(R(K_{9k+7},\Z_3)=9k+10\),
resolving a problem of Caro and Mifsud. 


\end{abstract}

\maketitle

\section{Introduction}
Classical Ramsey theory seeks monochromatic substructures in edge-colored complete graphs. Zero-sum Ramsey theory serarch for subgraphs whose labels sum to zero in 
edge-colored complete graphs with colors of a finite group. 
The subject was introduced by Bialostocki and  Dierker~\cite{BialostockiDierker1990}, partly as a graph-theoretic counterpart of the Erd\H{o}s--Ginzburg--Ziv
theorem~\cite{ErdosGinzburgZiv1961}; see Caro's survey~\cite{Caro1996} for an overview and Caro and Mifsud~\cite{CaroMifsud2026} for a recent systematic
study modulo~$3$.

Let $H$ be a finite simple graph and let $q\ge 2$ be an integer such that $q\mid e(H)$.  The zero-sum Ramsey number $R(H,\Z_q)$ is the least integer $N$ such that every edge labeling
\(w\colon E(K_N)\longrightarrow \Z_q\) of $K_N$
contains a copy $H'$ of $H$ satisfying
\(\sum_{e\in E(H')} w(e)=0\) in \(\Z_q.\)
The divisibility assumption implies that a monochromatic copy of $H$ is automatically zero-sum. Therefore, the classical multicolor Ramsey theorem ensures that $R(H,\Z_q)$ is finite.

We focus on complete graphs over $\Z_3$. The parameter $R(K_n,\Z_3)$ is admissible precisely when
\(3\mid \binom n2,\) equivalently, \(n\equiv 0,1\pmod 3.\)
Early general results for complete graphs were obtained by Caro~\cite{Caro1992} and by Alon and Caro~\cite{AlonCaro1993}. The exact-value problem then developed through a sequence of small cases and infinite congruence classes. Chung and Graham's theorem in ~\cite{ChungGraham1983} on precisely colored triangles gives $R(K_3,\Z_3)=11$. 
Harborth and Piepmeyer~\cite{HarborthPiepmeyer1994} proved
\(R(K_4,\Z_3)=7\)
and
\(R(K_6,\Z_3)=10.\)
They~\cite{HarborthPiepmeyer1996} subsequently established
\(R(K_n,\Z_3)=n+3\)
for \(n\ge 4\) with \(n\equiv 1,4\pmod 9,\)
and
\(R(K_n,\Z_3)=n+4\)
for \(n\ge 9\)  with \(n\equiv 0\pmod 9\). 
Caro later in~\cite{Caro1997} proved that
\(R(K_n,\Z_3)=n+4\)
for every $n\ge 6$ with \(n\equiv 0\pmod 3,\)
and also determined the exceptional small value
\(R(K_7,\Z_3)=11.\)

Consequently, before the present work, the exact value was known for every nontrivial admissible complete graph except the family
\( n\equiv 7\pmod 9\) for \(n\ge 16.\)
Caro and Mifsud singled this out as the sole remaining complete-graph problem modulo~$3$ (see~\cite[Problem~3]{CaroMifsud2026}).
\begin{problem}
Determine for $n\equiv 7\pmod 9$ the exact value of $R(K_n,\Z_3)$ which is either $n+3$ or $n+4$.
\end{problem}

Our result settles this family and, in fact, proves a uniform statement for large $n\equiv 1\pmod 3$.

\begin{theorem}\label{thm:main}
Let $n\ge 10$ and $n\equiv 1\pmod 3$.  Then
\(R(K_n,\Z_3)=n+3.\)
\end{theorem}

In particular, Theorem~\ref{thm:main} yields the following answer to the problem of Caro and Mifsud.

\begin{corollary}\label{cor:9k+7}
For every integer $k\ge 1$,
\(R(K_{9k+7},\Z_3)=9k+10.\)
\end{corollary}

Combining Theorem~\ref{thm:main} with the results recalled above gives the complete classification for all nontrivial admissible complete graphs.

\begin{corollary}\label{cor:classification}
For every integer $n\ge 3$ such that $3\mid\binom n2$,
\[
R(K_n,\Z_3)=
\begin{cases}
11, & n\in\{3,7\},\\
n+4, & n\ge 6\text{ and }n\equiv 0\pmod 3,\\
n+3, & n\ge 4,\ n\equiv 1\pmod 3,\text{ and }n\ne 7.
\end{cases}
\]
\end{corollary}

Throughout the proof, all algebraic identities involving edge weights are understood in $\Z_3$.  If $w$ is an edge labeling and $X$ is a vertex set, let $K_X$ denote the clique induced by $X$. Define
\[
  w(K_X):=\sum_{e\in E(K_X)}w(e).
\]
For disjoint vertex sets $X$ and $Y$, write $w(X,Y)$ for the sum of the weights of the edges joining $X$ and $Y$. The \textit{weighted degree} of a vertex $v$ is
\(
  d_w(v):=\sum_{u\ne v}w(uv).
\)

The proof is short. For the lower bound, we use a $0$--$1$ labeling of $K_{n+2}$ whose support is a $4$-regular graph. For the upper bound, we transform an arbitrary labeling of $K_{n+3}$ into one having constant weighted degree. A triangle lemma then supplies a triangle whose transformed
weight is prescribed. The complement of this triangle is the required zero-sum $K_n$. 

Section~2 proves the lower bound, and Section~3 proves the upper bound.

\section{The Lower Bound}
The lower-bound construction works for every admissible $n$ in the
congruence class $1\pmod 3$.

\begin{lemma}\label{lem:lower}
If $n\ge 4$ and $n\equiv 1\pmod 3$, then
\(
  R(K_n,\Z_3)\ge n+3.
\)
\end{lemma}

\begin{proof}
Set $M=n+2$.  Then $M\equiv 0\pmod 3$.  Label the vertices of $K_M$ by
$\Z_M$, and let $H$ be an $M$-vertex $4$-regular graph (such as the $2$-th power of the cycle $C_M$, in which $i$ and $j$ are
adjacent precisely when
\(|i-j|\in\{1,2\}\pmod M\)).
Define
\[
  w(e)=
  \begin{cases}
    1,& e\in E(H),\\
    0,& \text{otherwise}.
  \end{cases}
\]
Every vertex has weighted degree
\(
  d_w(v)=4=1
\)  
in $\Z_3,$
and the total weight is
\[
  W:=w(K_M)=e(H)=\frac{4M}{2}=2M=0.
\]

Let $S$ be any $n$-vertex subset of $V(K_M)$.  Since $M=n+2$, there are
distinct vertices $x,y$ such that
\(
  S=V(K_M)\setminus\{x,y\}.
\)
Deleting $x$ and $y$ removes total weight
$d_w(x)+d_w(y)-w(xy)$, because the edge $xy$ is counted twice in the
two weighted degrees.  Consequently,
\[
  \begin{aligned}
  w(K_S)
    &=W-d_w(x)-d_w(y)+w(xy)
    =0-1-1+w(xy)
    =1+w(xy).
  \end{aligned}
\]
As $w(xy)\in\{0,1\}$, the value of $w(K_S)$ is either $1$ or $2$, and
in particular is never zero.  Hence this labeling of $K_{n+2}$
contains no zero-sum copy of $K_n$, proving the lemma.
\end{proof}

\section{The Upper Bound}

The upper bound rests on the following triangle lemma.

\begin{lemma}\label{lem:triangle}
Let $N\ge 11$ and $N\equiv 1\pmod 3$.  Suppose that
\(
  w\colon E(K_N)\longrightarrow \Z_3
\)
has constant weighted degree, i.e., there is a constant $c\in\Z_3$ such that
\(d_w(v)=c\) for every \(v\in V(K_N).\)
Then there is a triangle $T$ such that
\(
  w(K_T)=c.
\)
\end{lemma}

\begin{proof}
Let \(W:=\sum_{e\in E(K_N)}w(e)\).
Let $\cT$ be the set of triangles of $K_N$, and for \(T\in\cT\), define \(w(T):=w(K_T).\)
Since $N=1$ in $\Z_3$, the weighted handshaking identity gives
\[
  2W=\sum_{v\in V(K_N)}d_w(v)=Nc=c.
\]
Since every edge belongs to exactly $N-2$ triangles, we have
\begin{equation}\label{eq:firstmoment}
  \sum_{T\in\cT}w(T)
   =(N-2)W=2W=c.
\end{equation}

We next compute the second moment.  Let $\mathcal A$ be the set of
unordered pairs $\{e,f\}$ of distinct adjacent edges of $K_N$.  On
expanding the squares of the triangle weights, every edge-square occurs
in $N-2$ triangles and every pair in $\mathcal A$ occurs in exactly one
triangle.  Hence
\begin{equation}\label{eq:triangle-second}
  \sum_{T\in\cT}w(T)^2
   =(N-2)\sum_e w(e)^2
     +2\sum_{\{e,f\}\in\mathcal A}w(e)w(f).
\end{equation}
On the other hand, expanding the squares of the weighted degrees gives
\begin{equation}\label{eq:degree-second}
  \sum_v d_w(v)^2
   =2\sum_e w(e)^2
     +2\sum_{\{e,f\}\in\mathcal A}w(e)w(f).
\end{equation}
Subtracting \eqref{eq:degree-second} from
\eqref{eq:triangle-second}, we obtain
\[
  \sum_{T\in\cT}w(T)^2-\sum_v d_w(v)^2=(N-4)\sum_e w(e)^2.
\]
Since $N-4=0$ in $\Z_3$, it follows that
\begin{equation}\label{eq:secondmoment}
  \sum_{T\in\cT}w(T)^2=\sum_v d_w(v)^2=Nc^2=c^2.
\end{equation}

It remains to distinguish the three possible values of $c$.

Suppose first that $c=1$ and that no triangle has weight $1$.  Then
$w(T)\in\{0,2\}$ for every $T$, and $x^2=-x$ for
$x\in\{0,2\}\subset\Z_3$.  Equations \eqref{eq:firstmoment} and
\eqref{eq:secondmoment} yield
\[
  1=c^2=\sum_Tw(T)^2=-\sum_Tw(T)=-c=2,
\]
a contradiction.

Suppose next that $c=2$ and that no triangle has weight $2$.  Then
$w(T)\in\{0,1\}$ for every $T$, and $x^2=x$ on this set.  Therefore
\[
  1=c^2=\sum_Tw(T)^2=\sum_Tw(T)=c=2,
\]
again a contradiction.

Finally, let $c=0$. A triangle has label sum zero in $\Z_3$ precisely when
its three labels are all equal or all distinct. Chung and Graham~\cite{ChungGraham1983} proved that every $3$-edge-coloring of $K_{11}$ contains such a triangle. Since $N\ge 11$, the labeling $w$
contains a zero-sum triangle, whose weight is $0=c$.

Thus in every case there is a triangle $T$ with $w(K_T)=c$.
\end{proof}

Now, we are ready to prove the upper bound.

\begin{lemma}\label{lem:upper}
If $n\ge 10$ and $n\equiv 1\pmod 3$, then
\(
  R(K_n,\Z_3)\le n+3.
\)
\end{lemma}

\begin{proof}
Set $N=n+3$.  Then $N\equiv 1\pmod 3$ and $N\ge 13$.  Consider an
arbitrary labeling
\(w\colon E(K_N)\longrightarrow\Z_3.\)
Let
\(d(v)=d_w(v)\) and \( W=w(K_N).\)
Define an auxiliary labeling $g$ on the same complete graph by
\begin{equation}\label{eq:def-g}
  g(uv)=w(uv)+d(u)+d(v).
\end{equation}
For a vertex $v$, we have
\[
  \begin{aligned}
  d_g(v)
    &=\sum_{u\ne v}\bigl(w(uv)+d(u)+d(v)\bigr)
    =d(v)+\bigl(2W-d(v)\bigr)+(N-1)d(v)
    =2W+(N-1)d(v).
  \end{aligned}
\]
Since $N-1=0$ in $\Z_3$, every vertex has the same $g$-weighted degree, namely
\(d_g(v)=2W.\)
By Lemma~\ref{lem:triangle}, there is a three-element vertex set $T$ such that
\begin{equation}\label{eq:chosen-triangle}
  g(K_T)=2W.
\end{equation}

Let
\(S=V(K_N)\setminus T.\)
Then $|S|=N-3=n$.  It remains to show that $K_S$ is zero-sum under the original labeling $w$.
Let
\(D_T=\sum_{t\in T}d(t).\)
Every edge inside $T$ is counted twice in $D_T$, while every edge between $T$ and $S$ is counted once.  Therefore
\[
  D_T=2w(K_T)+w(T,S),
\]
and hence
\[
  w(T,S)=D_T-2w(K_T).
\]
Using the decomposition of the total edge set into $E(K_T)$,
$E(K_S)$, and the cut between $T$ and $S$, we obtain
\begin{equation}\label{eq:complement-weight}
  \begin{aligned}
  w(K_S)
    &=W-w(K_T)-w(T,S)=W+w(K_T)-D_T.
  \end{aligned}
\end{equation}
On the other hand, each $d(t)$ occurs twice when \eqref{eq:def-g} is
summed over the three edges of $K_T$.  Thus
\begin{equation}\label{eq:triangle-transform}
  g(K_T)=w(K_T)+2D_T=w(K_T)-D_T.
\end{equation}
Combining \eqref{eq:complement-weight} and
\eqref{eq:triangle-transform} gives the identity
\begin{equation}\label{eq:key-identity}
  w(K_S)=W+g(K_T).
\end{equation}
Finally, \eqref{eq:chosen-triangle} and \eqref{eq:key-identity} imply
\[
  w(K_S)=W+2W=0.
\]
Hence every labeling of $K_{n+3}$ contains a zero-sum copy of $K_n$.
\end{proof}

\begin{proof}[Proof of Theorem~\ref{thm:main}]
The lower bound follows from Lemma~\ref{lem:lower}, and the upper bound follows
from Lemma~\ref{lem:upper}. Therefore
\[
  R(K_n,\Z_3)=n+3
\]
for every $n\ge 10$ with $n\equiv 1\pmod 3$.
\end{proof}

\end{document}